\documentclass[11pt]{article}  
\usepackage[utf8]{inputenc}
\usepackage[pagebackref=true]{hyperref}

\usepackage{amssymb,upgreek,amsmath,bm,bbm,amsthm, graphicx, epstopdf, tikz,xcolor,tocbibind}
\usepackage{algpseudocode,algorithm,algorithmicx}
\usepackage[justification=centering]{subfig}
\usepackage[export]{adjustbox}
\usepackage{mystyle} % contains notations and macros
\graphicspath{{figures/}}
\usepackage[OT2,T1]{fontenc}
\DeclareSymbolFont{cyrletters}{OT2}{wncyr}{m}{n}
\DeclareMathSymbol{\Sha}{\mathalpha}{cyrletters}{"58}
\usepackage{geometry}

\usepackage[page]{appendix}
\definecolor{DarkBlue}{rgb}{0.15,0.15,0.55}
\hypersetup{linktocpage=true,colorlinks=true,allcolors=DarkBlue}
\usepackage[capitalise]{cleveref}
\usepackage{titlesec}
\titleformat*{\paragraph}{\itshape\mdseries}

\title{TV-based Spline Reconstruction with Fourier Measurements: Uniqueness and Convergence of Grid-Based Methods}
\author{Thomas Debarre, Quentin Denoyelle, and Julien Fageot}
\begin{document} 

\maketitle

%ToDo V2: References to add~\cite{donoho1992superresolution,}

% !TEX root = ../main.tex

\begin{abstract}
We study the problem of recovering piecewise-polynomial periodic functions from their low-frequency information.
This means that we only have access to possibly corrupted versions of the Fourier samples of the ground truth up to a maximum cutoff frequency $K_c$. The reconstruction task is specified as an optimization problem with total-variation (TV) regularization (in the sense of measures) involving the $M$-th order derivative regularization operator $\Op L = \Op D^M$. The order $M \geq 1$ determines the degree of the reconstructed piecewise polynomial spline, whereas the TV regularization norm, which is known to promote sparsity, guarantees a small number of pieces. We show that the solution of our optimization problem is always unique, which, to the best of our knowledge, is a first for TV-based problems. Moreover, we show that this solution is a periodic spline matched to the regularization operator $\Op L$ whose number of knots is upper-bounded by $2 K_c$. We then consider the grid-based discretization of our optimization problem in the space of uniform $\Op L$-splines. On the theoretical side, we show that any sequence of solutions of the discretized problem converges uniformly to the unique solution of the gridless problem as the grid size vanishes. Finally, on the algorithmic side, we propose a B-spline-based algorithm to solve the grid-based problem, and we demonstrate its numerical feasibility experimentally. On both of these aspects, we leverage the uniqueness of the solution of the original problem.

%We show that, for elliptic regularization operators (\emph{e.g.}, the derivatives of any order), uniqueness is always guaranteed.
%To achieve this goal, we provide a new analysis of constrained optimization problems over Radon measures. 
%We demonstrate that either the solutions are always made of Radon measures of constant sign, or the solution is unique.
%Doing so, we identify a general sufficient condition for the uniqueness of the solution of a constrained optimization problem with TV-regularization, expressed in terms of the Fourier samples.
%Moreover, we provide a B-spline-based algorithm for the reconstruction of periodic functions and show that, in many cases, the  proposed discretized reconstruction uniformly converges to the unique solution of the continuous-domain problem as the grid size vanishes. 
%\ju{réécrire presque complètement. séparer les deux contributions : mesures pures et opérateurs.}
\end{abstract}

% !TEX root = ../main.tex

\section{Introduction}

In recent years, total-variation regularization techniques for continuous-domain inverse problems have shown to be very fruitful, with rapidly growing theoretical developments~\cite{Candes2013Super,Bredies2013inverse,Duval2015exact,Unser2017splines}, algorithmic progress~\cite{boyd2017alternating,denoyelle2019sliding,flinth2019linear}, and data science applications~\cite{courbot2019sparse,simeoni2020functional,courbot2020fast}.
%As is well known for its discrete-domain counterpart (\emph{i.e.}, $\ell_1$ optimization), this leads to variational problems whose solutions are not necessarily unique.

In this work, we study the reconstruction of an unknown periodic real function $f_0 : \mathbb{T} \rightarrow \mathbb{R}$ from the knowledge of its possibly noise-corrupted low-frequency Fourier series coefficients, where $\mathbb{T} = \R / 2\pi \Z = [0,2 \pi]$ is the torus whose two points $0$ and $2 \pi$ are identified. Let $K_c \geq 0$ be the cutoff frequency; we therefore have access to
\begin{equation} \label{eq:y}
    \bm{y} = (y_0, y_1, \ldots , y_{K_c}) \in \R \times \C^{K_c},
\end{equation}
where $y_k \approx \widehat{f}_0[k]$, the $k$th Fourier series coefficient of $f_0$. Note that, since $f_0$ is a real function, $y_0 \in \R$ is approximately  the mean $\widehat{f}_0[0] =  \langle f_0, 1\rangle$ of $f_0$, while $y_k \in \C$ for $k\neq 0$. Moreover, the Fourier series of $f_0$ is Hermitian symmetric, meaning that $\widehat{f}_0[-k] = \overline{\widehat{f}_0[k]} \in \C$ for every $k \in \Z$. The observation vector $\bm{y} $ in~\eqref{eq:y} has $2K_c + 1$ (real) degrees of freedom: one for the real mean $y_0$ and two for each other complex Fourier series coefficients in $\mathbb{C}$.
%Finally, we will always assume that $y_0 \geq 0$ (otherwise, we can consider the reconstruction of $-f_0$ instead of $f_0$). 

\subsection{Reconstruction via TV-based Optimization}
    
The recovery of a periodic function from finitely many observations is clearly an ill-posed problem. We choose to formulate the reconstruction task as a regularized optimization problem with a sparsity prior. More precisely, the reconstruction $f^*$ of $f_0$ is the solution of
\begin{equation} \label{eq:firstpb}
f^* \in \underset{f}{\arg\min} \  E(\bm{y}, \bm{\nu}(f))  + \lambda \lVert \Lop f \rVert_{\mathcal{M}},
\end{equation}
where $\bm{y} \in \R \times \C^{K_c}$ is the observation vector; $\bm{\nu}(f)$ is the measurement vector 
\begin{equation} \label{eq:nuf}
    \bm{\nu}(f) = (\widehat{f}[0] , \widehat{f}[1], \ldots , \widehat{f}[K_c] ) \in  \R \times \C^{K_c};
\end{equation}
$E(\V y , \cdot) : \R\times \C^{K_c} \rightarrow \R^+ \cup \{\infty\}$ is a data-fidelity functional which is a proper convex function, strictly convex over its effective domain\footnote{The effective domain of a convex function $g: X \rightarrow \R^+ \cup \{\infty\}$ is the set $\{x \in X, \ g(x) < \infty\}$~\cite{rockafellar1970convex}.}, lower semi-continuous (lsc), and coercive; $\lVert \cdot \rVert_{\mathcal{M}}$ is the total-variation (TV) norm on periodic Radon measures; and $\Lop$ is a regularization operator acting on periodic functions. For the sake of clarity, we focus on derivative operators of any order, \ie $\Lop = \Op D^M$ where $M \geq 1$, although our results can be extended to more general classes of operators such as fractional derivatives.
%The domain on which the functions $f$ are taken in \eqref{eq:firstpb} will be characterized thereafter.
%{\color{red}Question: what are the assumptions we fix on $E$? Strictly convex in its domain, admitting infinite values? Separability? Or at least a condition ensuring to separate between the null space and measure parts of the reconstruction for elliptic operators? It could be: $E (\bm{y}, \bm{z} ) = \sum_{k=0}^K E_k( y_k , z_k )$ but separability is more than what we really need. + we need $E_0(y_0, \cdot)$ to reach its mininmum at $y_0$.}

The data-fidelity term encourages the measurement vector $\bm{\nu}(f)$ to be close to the observations $\bm{y}$. A typical example is the quadratic functional
\begin{equation}\label{eq:datafidel}
E(\bm{y}, \bm{\nu}(f))  = \frac{1}{2} \lVert \bm{y} - \bm{\nu}(f) \rVert_2^2 = \frac{1}{2} \sum_{k=0}^{K_c} |y_k - \widehat{f}[k]|^2.
\end{equation}
The data fidelity~\eqref{eq:datafidel} is well-suited to an additive noise model where the measurements $\bm{y}$ are generated as $\bm{y} = \bm{\nu} ( f_0 ) + \bm{n}$ with $\bm{n}$ a complex Gaussian vector (see~\cite[Section IV-B]{badoual2018periodic} for more details).
Another case of interest is the indicator function $E(\bm{y}, \bm{\nu}(f)) = 0$ if $\bm{y} = \bm{\nu}(f)$ and $\infty$ otherwise, which leads to the constrained optimization problem\footnote{In this case, the value of the regularization parameter $\lambda > 0$ plays no role.} of the form
\begin{equation} \label{eq:constrained_pb_intro}
\underset{f, \ \bm{\nu}(f) =\bm{y}}{\arg\min} \lVert \Lop f \rVert_{\mathcal{M}}.
\end{equation}
Other classical data-fidelity functionals can be found in~\cite[Section 7.5]{simeoni2020functional}.

The choice of the total-variation norm promotes sparse and adaptive continuous-domain reconstruction, and has recently received a lot of attention (see Section~\ref{sec:relatedworks}). The operator $\Lop$ specifies the transform domain in which sparsity is enforced together with the regularity properties of the recovery.   %: Dirac recovery corresponds to $\mathrm{L} = \mathrm{Id}$ ~\cite{Candes2013Super}, and higher-order operators induce smoother reconstructions~\cite{Unser2017splines}.
In the absence of a regularization operator $\Op L$, Problem~\eqref{eq:firstpb} leads to the recovery of periodic Dirac masses; this scenario is the subject of our previous paper \cite{debarre2022uniqueness}. In the latter, we thoroughly analyze the cases of uniqueness of the solution of the resulting optimization problem. In particular, we show that unlike in this manuscript, both cases (uniqueness and non-uniqueness) can occur, and we gave a necessary and sufficient condition on the data vector $\V y$ that guarantees uniqueness.

\subsection{Contributions}
    
Problems of the form \eqref{eq:firstpb} have previously been studied in \cite{fageot2020tv} in a more general setting; the representer theorem from that paper guarantees the existence of a solution, without adjudicating on its uniqueness. Moreover, it gives the form of the extreme-point solution(s) as periodic $\Op L$-splines, \ie functions $f$ such that
    \begin{equation}
    \label{eq:spline_intro}
        \Lop f = \sum_{n=1}^N a_n \Sha(\cdot - x_n)
    \end{equation}
is a finite sum of shifted Dirac combs, the distinct Dirac locations $x_n$ being the knots of the spline (see Definition~\ref{def:Lspline}). Moreover, known proof techniques \cite{Unser2017splines, fageot2020tv, fisher1975spline} allow us to show that the number of knots $N$ is bounded by $N \leq 2 K_c +1$. 
%, as exemplified for Dirac recovery (corresponding to the identity operator $\Lop = \mathrm{I}$)~\cite{Candes2013Super}. 
Our contributions can be detailed as follows. \\

    \textit{(i) Uniqueness of the Solution.}
    Our main result is Theorem~\ref{theo:main}, in which we prove that
    %in our particular setting (with Fourier-domain measurements and a high-order derivative operator $\Op L$)
    the solution to Problem~\eqref{eq:firstpb} is always unique. Moreover, we slightly improve the upper bound on the number of knots to $K \leq 2 K_c$ in \eqref{eq:spline_intro}. Our proof relies on a result of our previous paper \cite[Corollary 1]{debarre2022uniqueness}. To the best of our knowledge, Theorem~\ref{theo:main} is the first systematic uniqueness result for the analysis of TV-based variational problems such as~\eqref{eq:firstpb}. This result has both theoretical and algorithmic implications, which we leverage in our other contributions.\\
    
   \textit{(ii) Uniform Convergence of Grid-Based Methods.}
    We study the grid-based discretization of Problem~\eqref{eq:firstpb}. More precisely, we restrict its search space to the finite-dimensional space of uniform $\Op L$-splines, \ie $\Op L$-splines whose knots lie on a uniform grid. We show that as the grid gets finer, any sequence of solutions of the discretized problems converges in uniform norm towards the unique solution of Problem~\eqref{eq:firstpb}. This form of convergence is remarkably strong: in particular, it implies convergence for any $L_p$ norm with $1 \leq p \leq \infty$. \\ %Existing related results for grid-based methods typically use tools from $\Gamma$-convergence \cite[Proposition 4]{duval2017sparseI} and exhibit weak$^*$-type convergence. \\
    
    \textit{(iii) Grid-Based Algorithm.}
    We propose a periodic adaptation of the B-spline-based algorithm developed in~\cite{Debarre2019} to solve Problem~\eqref{eq:firstpb}. Thanks to our aforementioned uniform convergence result, the reconstructed signal is guaranteed to be uniformly close to the gridless solution when the grid is sufficiently fine. We provide some experimental results of our algorithm on some simulated data that demonstrate its numerical feasibility.
  %  {\color{blue}\textit{(iii) A new sliding Frank-Wolfe algorithm for TV-regularization with elliptic operator:} 
  %  We introduce a reconstruction algorithm which is able to recover the unique solution to~\eqref{eq:firstpb} for an elliptic operator~$\Lop$. 
 %   This requires to adapt the Frank-Wolfe algorithm developed in~\cite{denoyelle2019sliding} to this setting. This adaptation is not obvious. For instance, the extreme-points of the unit ball of the native space $\mathcal{M}_{\Lop}(\T)$ of \eqref{eq:firstpb} (see Section \ref{sec:nativeetc}) are periodic $\Lop$-splines with two knots, which significantly impacts the algorithm.
 %   Our analysis includes theoretical guarantees for perfect reconstruction in finite time together with an empirical demonstration of the relevance of our approach on simulations.}

    \subsection{Related Works}\label{sec:relatedworks}

        \paragraph{Optimization over Radon measures:}
        
        There exists a vast literature concerned with Dirac recovery using the TV norm as a regularizer, that is, problems of the form~\eqref{eq:firstpb} in the absence of a regularization operator $\Op L$. Some of the more recent works concerning this topic include \cite{Candes2013Super, Bredies2013inverse, flinth2019linear, duval2017sparseI, de2012exact, candes2014towards,Azais2015Spike,duval2017sparseII,poon2019support}. However, our focus in this paper is on generalized TV regularization with a nontrivial operator $\Op L$. We refer to the introduction of our previous paper \cite{debarre2022uniqueness} for a more detailed coverage of the Dirac-recovery literature.

        \paragraph{From sparse measures to splines and beyond:}
          % The study of optimization problems of the form \eqref{eq:secondpb} can be traced back to the pioneering works of Beurling~\cite{beurling1938integrales}, where Fourier-sampling measurements were also considered.
        In recent years, several works have extended the TV-based Dirac recovery framework to smoother continuous-domain signals by considering generalized total-variation regularization, \emph{i.e.}, problems such as~\eqref{eq:firstpb} with a nontrivial regularization operator $\Lop$. In~\cite{Unser2017splines}, Unser \etal revealed the connection between the constrained problem \eqref{eq:firstpb} (in a non-periodic setting) and spline theory for general measurement functionals: the extreme-point solutions are necessarily $\Lop$-splines. This result was revisited, extended, and refined by several authors     ~\cite{simeoni2020functional,fageot2020tv,gupta2018continuous,aziznejad2021multikernel,boyer2019representer,bredies2019sparsity,flinth2019exact,Unser2019native,filbirsuper}.
        This manuscript will strongly rely on the periodic theory of TV-based optimization problem recently developed in~\cite{fageot2020tv}.
        
        The purpose of most of these works is to describe the solution sets of certain relevant optimization problems, which are typically nonunique. However, in our setting, as we show in Theorem~\ref{theo:main}, the solution to Problem~\eqref{eq:firstpb} is always unique. The closest work in this direction is our recent paper~\cite{debarre2020sparsest}, where we provide a full description of the solution set of nonperiodic TV-based optimization problems with a regularization operator $\Lop = \mathrm{D}^2$ (which leads to piecewise-linear reconstructions), and spatial sampling measurements $\V \nu$. This study includes the characterization of the cases of uniqueness~\cite[Proposition 6 and Theorem 2]{debarre2020sparsest}, which, contrary to Problem~\eqref{eq:firstpb}, is not systematic.
       % Our manuscript is, to the best of our knowledge, the first systematic uniqueness study for TV-optimization problems with Fourier sampling measurement. 

        \paragraph{Convergence results and algorithms for discretized problems.}
        
       The convergence of discretized optimization schemes to the solutions of continuous-domain TV-regularized problems has been studied by several authors, such as~\cite{Bredies2013inverse,Duval2015exact,denoyelle2019sliding,flinth2019linear,aubel2018theory}. Grid-based methods have specifically been considered in~\cite{Duval2015exact,duval2017sparseI,duval2017sparseII}. In these works, the authors prove convergence results in the weak* sense, which is adapted to the space of Radon measures, in a setting where no systematic uniqueness results are known.
        To the best of our knowledge, our work is the first to prove the convergence of solutions of discretized generalized TV-based problems towards the solution of the original problem, let alone in a strong sense such as the uniform norm. To achieve this, we leverage our uniqueness result of Theorem~\ref{theo:main}. On the algorithmic side, grid-based methods to solve optimization problems with TV-based regularization have been proposed in \cite{Debarre2019,gupta2018continuous,debarre2019hybrid,debarre2021continuous, llorensjover2022coupled}.
        
%        our work is the first one that considers the convergence of grid-based methods for spline reconstruction  that uses the systematic uniqueness results obtained for elliptic operators, and that provides a uniform convergence of the discretized solutions to the unique one of \eqref{eq:firstpb} as the grid size vanishes. 
  %      {\color{blue}\paragraph{Grid-based and gridless algorithms for measure and spline reconstruction:}
  %      In parallel to the development of the infinite-dimensional optimization theory over Banach measure spaces, significant efforts have been made to propose numerical algorithms to solve TV-based optimization problems.    
  %      This includes grid-based method for spline reconstruction~\cite{gupta2018continuous,Debarre2019}    
  %      and off-the-grid methods, that are mostly developed for Dirac stream reconstruction~\cite{courbot2019sparse,boyd2017alternating,flinth2019linear,denoyelle2019sliding}. In this paper, we use for a first time, to the best of our knowledge, the Frank-Wolfe algorithm for a spline-based reconstruction in the periodic setting.
   %     We also mention the development of semi-definite programming methods, that have been especially considered for low-frequency measurements~\cite{Candes2013Super,de2016exact,Catala2017low}.}
        
    \subsection{Outline}
    
    The paper is organized as follows. Section~\ref{sec:maths} introduces the mathematical material used in this paper. 
    In Section~\ref{sec:gtv_problem}, we present our optimization problem of interest \eqref{eq:firstpb} and prove that it always has a unique solution.
%    the implication for~\eqref{eq:firstpb} for invertible operator in Section~\ref{sec:solutioninvertibleop}, and the
 %   uniqueness of the solutions to~\eqref{eq:firstpb} for elliptic operators $\Lop$ in Section~\ref{sec:solutionuniquegeneral}. 
 %   {\color{blue}Section~\ref{sec:algo} presents our new sliding Frank-Wolfe algorithm for reaching the unique solution to~\eqref{eq:firstpb} for an elliptic regularization operator, together with its theoretical analysis and its application to simulations.}
 In Section~\ref{sec:grid_convergence}, we present our grid-based discretization of Problem~\eqref{eq:firstpb}, and prove that its solutions converge uniformly to that of the original problem when the grid size goes to zero. We present our method for solving this discretized problem using a B-spline basis in Section~\ref{sec:algo}. Finally, we exemplify our results on simulations in Section~\ref{sec:experiments}.

% !TEX root = ../main.tex

    \section{Mathematical Preliminaries}
    \label{sec:maths}
    
        \subsection{Periodic Functions and Periodic Splines} 
        \label{sec:opgrspline}
 
    We first introduce some notations and recall some basic facts concerning periodic functions and their Fourier series. More details can be found  in~\cite[Section 2]{fageot2020tv}. 
    The Schwartz space of infinitely smooth periodic function is denoted by $\Sch$, endowed with its usual Fr\'echet topology. Its topological dual is the space of periodic generalized functions $\Schp$. 
    
    For $k\in \Z$, let $e_k: \T \rightarrow \C$ be the complex exponential function $e_k(x) = \exp( \mathrm{i} k x)$, which is clearly in $\Sch$. 
    The Fourier series coefficients of $f \in \Schp$ are given by $\widehat{f}[k] = \langle f , e_k \rangle \in \C$. For a real function $f$, these coefficients are Hermitian symmetric, \emph{i.e.}, $\widehat{f}[-k] = \overline{\widehat{f}[k]}$ for all $k \in \Z$,
    which implies in particular that $\widehat{f}[0] \in \R$. We then have that 
    $f =  \sum_{k\in\Z} \widehat{f}[k] e_k$ for any $f \in \Schp$, where the convergence is in $\Schp$. 
    The Dirac stream is defined as $\Sha = \sum_{n\in\Z} \delta( \cdot - 2 \pi n)$. Its Fourier coefficients are $\widehat{\Sha}[k] = 1$ for any $k\in \Z$.
    %The Fourier sequence $(\widehat{f}[k])_{k\in \Z}$ of a periodic generalized function is bounded by a polynomial. 
        The derivative operator is denote by $\D$. More generally, we consider the $M$th-order derivative operator $\Lop = \D^M$ for a fixed integer $M\geq 1$.
    We then have that $ \Lop f = \sum_{k\in \Z} (\mathrm{i}  k)^M \widehat{f}[k] e_k$.  \\

    We define the \textit{periodic Green function} $g_{\Lop}$ of $\Lop = \D^M$ as the function
    \begin{equation}
    \label{eq:green_def}
        g_\Lop(x) = \sum_{k \neq 0} \frac{e_k}{(\mathrm{i} k)^M}.
    \end{equation}
    Then, $g_{\Lop}$ is the unique periodic and zero-mean function such that $\D^M g_\Lop = \Sha - 1$. It is worth noting that $g_\Lop$ is not a Green's function in the usual sense: there is no periodic function $g$ such that $\D^M g = \Sha$, since $\D^M g$ necessarily has zero mean, whereas the Dirac stream does not. See~\cite[Section 2.2]{fageot2020tv} for more details on this matter.
    
        \begin{definition}
        \label{def:Lspline}
       Let $M\geq 1$ and $\Lop = \D^M$. We say that $f$ is a \emph{periodic $\Lop$-spline} (or simply a $\Lop$-spline) if 
        \begin{equation} \label{eq:Lopf}
            \Lop f = w = \sum_{n=1}^N a_n \Sha(\cdot - x_n)
        \end{equation}
       where $N\geq 0$, $a_n \in \R\backslash\{0\}$, and the knots $x_n \in \T$ are pairwise distinct. We call $w$ the \emph{innovation} of the $\Lop$-spline $f$. 
        \end{definition}
        
        A function $f$ satisfies~\eqref{eq:Lopf} if and only if 
        \begin{equation}
        \label{eq:spline_green}
        f = a_0 + \sum_{n=1}^N a_n g_\Lop(\cdot - x_n)    
        \end{equation}
        for some $a_0 \in \R$. In this case, we necessarily have that $\sum_{n=1}
       ^N a_n = 0$. This is a particular case of~\cite[Proposition 2.8]{fageot2020tv} and simply follows from taking the mean (or $0$th Fourier coefficient) in~\eqref{eq:Lopf}, giving $0 = \widehat{L}[0] \widehat{f}[0] = \sum_{n=1}^N a_n$. It is worth noting that the Green's function $g_{\Lop}$ is \emph{not} a $\Lop$-spline. However, for any $x_0 \in \T \setminus \{0\}$, $g_{\Lop} - g_{\Lop}(\cdot - x_0)$ is a periodic $\Lop$-spline.
       
       A $\Lop$-spline is a periodic piecewise-polynomial function of degree at most $(M-1)$ and with $(M-2)$ continuous derivatives. The case $M=1$ corresponds to piecewise-constant functions, while $M=2$ leads to piecewise-linear continuous functions.

        \subsection{Periodic Radon Measures and Native Spaces}
        \label{sec:nativeetc}
        
        Let $\Rad$ be the space of periodic Radon measures. 
        By the Riesz-Markov theorem~\cite{gray1984shaping}, it is the continuous dual of the space $\mathcal{C}(\T)$ of continuous periodic functions endowed with the supremum norm. The total-variation norm on $\Rad$, for which it forms a Banach space, is given by
        \begin{equation} \label{eq:tvnorm}
            \lVert w \rVert_{\mathcal{M}} = \sup_{f \in \mathcal{C}(\T), \ \lVert f \rVert_\infty \leq 1} \langle w , f \rangle.
        \end{equation}
        We denote by $\mathcal{M}_0(\T)$ the set of Radon measures with  zero mean, \emph{i.e.} 
        $\mathcal{M}_0(\T) = \{w \in \mathcal{M}(\T), \ \widehat{w}[0] = 0\}$. It is the continuous dual of the space $\mathcal{C}_0(\T) = \{f \in \mathcal{C}(\T), \ \widehat{f}[0] = 0\}$ of continuous functions with zero mean.
        
        Let $\Lop = \D^M$ for some $M \geq 1$. We define the \emph{native space} associated to $\Lop$ as 
        \begin{equation}
            \mathcal{M}_{\Lop}(\T) = \{ f \in \Schp : \ \Lop f \in \Rad \}.
        \end{equation}
        Periodic native spaces have been studied for general spline-admissible operators (\emph{i.e.}, periodic operators with finite-dimensional null space and which admit a pseudoinverse) in~\cite[Section 3]{fageot2020tv}. Proposition~\ref{prop:periodicTVused} recalls some important properties of native spaces for the particular case of the $M$th order derivative operator. For this purpose, we define the pseudo-inverse operator  $\Lop^{\dagger}$ such that
        \begin{equation}
             \Lop^{\dagger} f = \sum_{k\neq 0} \frac{\widehat{f}[k]}{(\mathrm{i}k)^M} e_k
        \end{equation}
        for any $f \in \Spc S'(\T)$. In particular, we have that $\Lop^{\dagger} \Sha = g_\Lop$. 
        
        \begin{proposition}[Theorem 3.2 in \cite{fageot2020tv}]
        \label{prop:periodicTVused}
        Let $\Lop = \D^M$ for some $M \geq 1$. We have the direct-sum relation
        \begin{equation}
            \mathcal{M}_{\Lop}(\T) = \Lop^{\dagger} \mathcal{M}_0(\T) \oplus \Span\{1\},
        \end{equation}
         and any $f \in \mathcal{M}_{\Lop}(\T)$ has a unique decomposition as 
        \begin{equation}
            \label{eq:uniquedecompo}
            f = \Lop^{\dagger} w + a
        \end{equation}
        where $w \in \mathcal{M}_0(\T)$ and $a \in \R$ are given by $w = \Lop f$ and $a = \widehat{f}[0]$. Then, $\mathcal{M}_{\Lop}(\T)$ is a Banach space for the norm 
        \begin{equation}
            \lVert f \rVert_{\mathcal{M}_{\Lop}} = \lVert w \rVert_{\mathcal{M}} + |a|.
        \end{equation}
        \end{proposition}
        
  %  Note that the measurement functional $\bm{\nu}$ in \eqref{eq:nuf} is well defined over $\mathcal{M}_{\Lop}(\T)$, and more generally over $\Schp$, since complex exponentials are infinitely smooth. 

% !TEX root = ../main.tex

\section{Uniqueness of TV-Based Penalized Problems}
            \label{sec:gtv_problem}

        It is well known that TV-based optimization problems with regularization operators lead to splines solutions~\cite{Unser2017splines}. This is both an existence result and a \emph{representer theorem}, which provides the form of the (extreme-point) solutions of the optimization task. Here, we focus on problems of the form~\eqref{eq:firstpb}, whose main specificity compared to related works in the literature is the periodic setting and the Fourier-domain measurement operator $\V \nu$.

 We now state our main result, which guarantees the uniqueness of the solution to Problem~\eqref{eq:firstpb}. The proof relies on a result of our previous paper \cite[Corollary 1]{debarre2022uniqueness} by reformulating Problem~\eqref{eq:firstpb} over the space of Radon measures; interestingly, the regularization operator $\Op L = \Op D^M$ leads to systematic uniqueness, which is not true of generic problems formulated over Radon measures.  %Our result also (slightly) improves the upper bound on the number of knots of the solution from $2K_c + 1$ to $2K_c$ for our specific setting.
\begin{theorem} \label{theo:main}
        Let $\Lop = \Op D^M$ with $M \geq 1$, $K_c \geq 0$ be the cutoff frequency of the low-pass filter $\bm{\nu} : \mathcal{M}_{\Lop}(\T) \rightarrow \R \times \C^{K_c}$ defined in \eqref{eq:nuf}, $\bm{y} \in \R \times \C^{K_c}$, $E(\V y , \cdot) : \R\times \C^{K_c} \rightarrow \R^+$ be a functional which is a proper convex function, strictly convex over its effective domain, lsc, and coercive, and $\lambda > 0$.
        Then, the optimization problem
        \begin{equation}
        \label{eq:pbelliptic}
            \mathcal{V}_\lambda(\bm{y}) = 
            \underset{f\in \mathcal{M}_{\Lop}(\T)} {\arg \min} \  E(\bm{y}, \bm{\nu}(f))  + \lambda \lVert \Lop f \rVert_{\mathcal{M}},
        \end{equation}
        admits a unique solution which is a $\Lop$-spline whose number of knots is bounded by $2K_c$. 
\end{theorem}
\begin{proof}

        Using a classical argument based on the strict convexity of $E(\bm{y},\cdot)$ (see for instance~\cite[Proposition 7]{debarre2020sparsest}), we deduce that all solutions $f^\ast$ of Problem~\eqref{eq:pbelliptic} share an identical observation vector $\bm{y}_\lambda \in \R^M$, that is, $\forall f^\ast \in \mathcal{V}_{\lambda}(\bm{y})$, we have $\V\nu(f^\ast) = \bm{y}_\lambda$. Hence, Problem~\eqref{eq:pbelliptic} is equivalent to 
        \begin{equation} \label{eq:constrained_pb}
            \mathcal{V}_{\lambda}(\bm{y}) =  \underset{f \in \mathcal{M}_{\Lop}(\T), \ \bm{\nu}(f) = \bm{y}_{\lambda}}{\arg \min} \lVert \Lop f \rVert_{\mathcal{M}}.
        \end{equation}
               By Proposition~\ref{prop:periodicTVused}, any $f \in \mathcal{M}_{\Lop}(\T)$ admits a unique decomposition $f = \Lop^\dagger w + a \in \mathcal{M}_{\Lop}(\T)$ with $(w, a) \in\mathcal{M}_0(\T) \times \R$. By plugging in this expansion into the cost functional of Problem~\eqref{eq:constrained_pb}, we get that the latter is equivalent to 
        \begin{align}
        \label{eq:reformulation_measures_indicator}
            \underset{(w,a) \in \mathcal{M}_{0}(\T) \times \R, \ \bm{\nu}(\Lop^\dagger w + a) = \V y_\lambda}{\arg \min} \lVert w \rVert_{\mathcal{M}} \Longleftrightarrow  \underset{w \in \mathcal{M}_{0}(\T), \ \bm{\nu}(w) = \V z }{\arg \min} \lVert w \rVert_{\mathcal{M}},
        \end{align}
        where $\bm{z} \in \R\times \C^{K_c}$ is defined as $z_0 = 0$ and $z_k = \widehat{L}[k] {(\V y_\lambda)}_k$ for $k \neq 0$. The equivalence in \eqref{eq:reformulation_measures_indicator} comes from the fact that $\bm{\nu}(\Lop^{\dagger} w + a) = \left( a , \frac{\widehat{w}[1]}{\widehat{L}[1]} , \ldots, \frac{\widehat{w}[K_c]}{\widehat{L}[K_c]} \right)$. Any $f^* \in \mathcal{V}_{\lambda}(\bm{y})$ can thus be decomposed as $f^* = \Lop^\dagger w^* + {(\V y_\lambda)}_0$ where $w^*$ is a solution of Problem~\eqref{eq:reformulation_measures_indicator}.
        
   We now prove that Problem~\eqref{eq:reformulation_measures_indicator} has a unique solution $w^*$ which is a sum of at most $2 K_c$ Dirac impulses. If $\V z = \V 0$, then the result trivially holds, the unique solution being $w^\ast = 0$. We now assume that $\V z \neq \V 0$. In this case, Problem~\eqref{eq:reformulation_measures_indicator} has a nonzero optimal value due to the fact that $w=0$ cannot be a solution. Since $z_0 = 0$ and $\V z \neq 0$, Problem~\eqref{eq:reformulation_measures_indicator} satisfies the assumptions of \cite[Corollary 1]{debarre2022uniqueness}. Hence, the uniqueness of $w^*$ and the fact that it is a sum of at most $2K_c$ Dirac impulses. This in turn implies the uniqueness of the solution $f^* = \Lop^\dagger w^* + (\V y_\lambda)_0$ as well as the fact that it is an $\Op L$-spline with at most $2 K_c$ knots.
\end{proof}

\textit{Remark.} Theorem~\ref{theo:main} remains valid for more general operators $\Op L$, namely any spline-admissible operator in the sense of \cite[Definition 2]{fageot2020tv} whose null space includes constant functions, \ie $\Op L \{ 1 \} = 0$.

Theorem~\ref{theo:main} has three components: i) it guarantees the uniqueness of the solution, it provides ii) the form of the solution and iii) an upper bound on the number of knots of the solution. The first item, arguably the most striking one, is completely new; existing results typically provide the form of extreme-point solutions of the problem. We are not aware of any other systematic uniqueness results concerning inverse problems with TV-based regularization in the literature. The second item is already known; it has been proved for our setting in \cite[Theorem 4]{fageot2020tv}. Finally, concerning the third item, known proof techniques \cite{Unser2017splines, fageot2020tv, fisher1975spline} allow us to reach the bound $2 K_c + 1$, which we improve to $2 K_c$. 

One can actually be slightly more precise and show that the mean of the solution is known under very mild conditions on the cost functional $E$. Under this assumption, we also provide a reformulation of Problem~\eqref{eq:pbelliptic} over the space of Radon measures.
        
        \begin{proposition} \label{prop:y0forz0}
        We assume that we are under the conditions of Theorem \ref{theo:main} and that the data-fidelity cost functional $E$ is such that for any fixed $(z_1, \ldots , z_{K_c}) \in \C^{K_c}$, we have
        \begin{equation} \label{eq:y0argmin}
            y_0 = \argmin_{z_0 \in \R} \ E ( \bm{y}, \bm{z})
        \end{equation}
        where $\bm{y} = (y_0, y_1, \ldots , y_{K_c}) \in \R \times \C^{K_c}$ and $\bm{z} = (z_0, z_1, \ldots , z_{K_c}) \in \R \times \C^{K_c}$. Then, the unique solution $f^*$ to \eqref{eq:pbelliptic} admits the decomposition $f^* = y_0 + \Lop^\dagger w^*$ where
        \begin{align}
        \label{eq:reformulation_measures}
            w^* = \underset{w \in \mathcal{M}_{0}(\T)}{\arg \min} E( \V y, \bm{\nu}(\Lop^\dagger w + y_0) ) + \lambda \lVert w \rVert_{\mathcal{M}}.
        \end{align}
        In particular, this implies that $\widehat{f^*}[0] = y_0$.
        \end{proposition}
        
        \begin{proof}
        Similarly to our manipulation in \eqref{eq:reformulation_measures_indicator}, Problem~\eqref{eq:pbelliptic} is equivalent to 
        \begin{align}
        \label{eq:reformulation_measures_full}
            (w^*, a^*) = \underset{(w,a) \in \mathcal{M}_{0}(\T) \times \R}{\arg \min} E( \V y, \bm{\nu}(\Lop^\dagger w + a) ) + \lambda \lVert w \rVert_{\mathcal{M}},
        \end{align}
         with $f^* = \Lop^\dagger w^* + a^*$. Problem~\eqref{eq:reformulation_measures_full} has a unique solution due to that of Problem~\eqref{eq:pbelliptic} (proved in Theorem~\ref{theo:main}), and to the uniqueness of the decomposition of $f^*$ (Proposition~\ref{prop:periodicTVused}).
        Then, we have that $\bm{\nu}(a^* + \Lop^{\dagger} w^*) = ( a^* , \widehat{\Lop^\dagger w^*}[1], \cdots , \widehat{\Lop^\dagger w^*}[K_c])$, which by \eqref{eq:y0argmin} implies that $E( \bm{y}, \bm{\nu}( a^* + \Lop^\dagger w^* ) ) \geq E( \bm{y}, \bm{\nu}( y_0 + \Lop^\dagger w^* ) )$, with equality if and only if $a^* = y_0$. Hence, since the constant $a$ does not impact the regularization in \eqref{eq:reformulation_measures}, we must have that $a^* = \widehat{f^*}[0] = y_0$. Problem~\eqref{eq:reformulation_measures_full} can thus be rewritten as \eqref{eq:reformulation_measures}.
        \end{proof}

        \textit{Remark.} The relation \eqref{eq:y0argmin} holds for virtually all classical cost functionals, including any $\ell_p$ norm-based cost such as the quadratic data fidelity \eqref{eq:datafidel}, or any separable cost whose minimum over each component is reached when $y_m = z_m$, such as indicator functions. Proposition~\ref{prop:y0forz0} ensures that the mean of the solution of Problem~\eqref{eq:pbelliptic} is given by  $\widehat{f^*}[0] = y_0$. 
\section{Uniform Convergence of Grid-Based Methods}
\label{sec:grid_convergence}

A common way to solve infinite-dimensional continuous-domain problems such as \eqref{eq:penalized_pb_quadratic} algorithmically is to discretize them using a uniform finite grid \cite{duval2017sparseI, duval2017sparseII}. In this section, we propose such a discretization method of the problem
\begin{align}\label{eq:penalized_pb_quadratic}
f^* =	\underset{f\in\Spc M_{\Op L}(\T)}{\arg \min} \ \frac{1}{2} \norm{\V{y} - \nuf(f) }_2^2 + \la\mnorm{\Op L f},
\end{align}
that is, Problem~\eqref{eq:pbelliptic} with a quadratic data-fidelity cost $E(\V y, \V z) = \frac12 \Vert \V y - \V z \Vert^2_2$. Note that we no longer denote the solution of Problem~\eqref{eq:penalized_pb_quadratic} as a set but as a function $f^*$, since Theorem~\ref{theo:main} guarantees that this solution is unique. We restrict to the case of the quadratic data fidelity for the sake of simplicity, although our results hereafter hold for more general choices of $E$. Our choice clearly satisfies the assumption of Proposition~\ref{prop:y0forz0}, hence the solution $f^*$ of \eqref{eq:penalized_pb_quadratic} satisfies $\widehat{f}^*[0] = y_0$. 
%\cite[Table 1]{fageot2020tv}. 

Our discretization method, which was introduced for similar problems in~\cite{Debarre2019}, consists in restricting the search space of Problem~\eqref{eq:penalized_pb_quadratic} to the space of uniform $\Op L$-splines $\MLh$, \ie $\Op L$-splines in the sense of Definition~\ref{def:Lspline} with knots $x_n$ on a uniform grid. The space $\MLh$ is defined as follows for a grid size $h = \frac{2 \pi}{P}$, where $P \in \N$, $P \geq 1$, is the number of grid points:
\begin{align}
\label{eq:space_uniform_splines}
    \MLh = \left\{ f\in \Spc S'(\TT), \ \Op L f = \sum_{p=0}^{P-1} a[p] \Sha \left(\cdot -  \frac{2 \pi p}P \right) \right\}.
\end{align}

Our choice of restricting the search space of Problem~\eqref{eq:penalized_pb_quadratic} to $\MLh$ is guided by Theorem~\ref{theo:main}, which states that the unique solution to this problem is a $\Op L$-spline. Hence, this choice of space is compatible with the sparsity-promoting regularization $\Vert \Op L \cdot \Vert_\Spc M$. Although in general, the solution of our problem does not have knots on a uniform grid, it can be approximated arbitrary closely with an element of $\MLh$ when $P$ is large. The other main feature of our method is that the computations are exact in the continuous domain, both those of the forward model and of the regularization term. Our discretized optimization problem then becomes
\begin{align}
\label{eq:discretized_continuous_pb}
\Spc V_{\lambda, P}(\V y) =	\underset{f\in \MLh}{\argmin} \ \frac{1}{2} \norm{\V{y} - \nuf(f) }_2^2 + \la\mnorm{\Op L f}.
\end{align}
Note that contrary to the original Problem~\eqref{eq:penalized_pb_quadratic}, the solution set $\Spc V_{\lambda, P}(\V y)$ of the discretized problem~\eqref{eq:discretized_continuous_pb} is not necessarily unique. 

As we shall demonstrate in Section~\ref{sec:algo}, Problem~\eqref{eq:penalized_pb_quadratic} can be solved algorithmically with standard finite-dimensional solvers. However, the important question of how well it approximates the original Problem~\eqref{eq:penalized_pb_quadratic} still remains. We answer this question in Theorem~\ref{thm:unifconv} by proving that any sequence of elements of $\Spc V_{\lambda, P}(\V y)$ converge in a strong sense --- namely, uniform convergence --- towards $f^*$ when $P \to \infty$.

\begin{theorem}
\label{thm:unifconv}
Let $\Op L = \Op D^M$ with $M \geq 2$, $\bm{y} \in \R^+ \times \C^{K_c}$, and $\lambda > 0$. We denote by $f^*$ the unique solution to \eqref{eq:penalized_pb_quadratic}. For any $P \geq 1$, we set $f_P^* \in \Spc V_{\lambda, P}(\V y)$. Then, we have that
\begin{equation} \label{eq:unifconvmoica}
\lVert f^* - f_P^* \rVert_\infty  \underset{P\rightarrow \infty}{\longrightarrow} 0. 
\end{equation}
\end{theorem}

\textit{Remark 1.} Despite the fact that the solutions to \eqref{eq:discrete_pb} may not be unique, Theorem~\ref{thm:unifconv} ensures that the convergence \eqref{eq:unifconvmoica} holds for any choice of the $f_P^*$. 

\textit{Remark 2.} Uniform convergence implies convergence with respect to any $L_p$ norm for $1 \leq p \leq \infty$, since we have $\Vert f \Vert_p \leq (2 \pi)^{1/p} \Vert f \Vert_\infty$ for any $f \in \Spc M_\Op L(\T)$.

\textit{Remark 3.}
Theorem~\ref{thm:unifconv} holds for more general settings than Problem~\eqref{eq:penalized_pb_quadratic}. More specifically, our proof seamlessly extends to the more general setting of Theorem~\ref{theo:main} for any cost functional $E$ that is continuous with respect to its second argument, such as $\ell_p$ losses of the form $E(\V y, \V z) = \Vert \V y - \V z \Vert_p^p$. Compared to the setting of Theorem~\ref{theo:main}, this notably excludes indicator functions, \ie the constrained optimization Problem~\eqref{eq:constrained_pb}. Concerning the regularization operator $\Op L$, Theorem~\ref{thm:unifconv} readily extends to any operator $\Op L$ such that $\Op L \{ 1 \}$ and whose periodic Green's function is Lipschitz. This notably excludes the case $\Op L = \Op D$, \ie $M=1$.

%We may consider more general data-fidelity cost functionals $E$. For our theoretical result in \eqref{eq:penalized_pb_quadratic}: our solver, the alternating direction method of multipliers (ADMM) \cite{boyd2010distributed}, only requires that $E$ be differentiable or proximable. 

\begin{proof}
We first introduce
%\begin{align}
   $\mathcal{M}_{0,\frac{2\pi}{P}(\T)} =
    \{ w \in \mathcal{M}_0(\T), \ w = \sum_{p=0}^{P-1} a[p] \Sha (\cdot -  \frac{2 \pi p}P) \}$, the uniform discretization of $\mathcal{M}_0(\T)$ using Dirac impulses. 
    %, \nonumber  \\
%    \mathcal{M}_{\mathrm{D}^M, \frac{2\pi}{P}}(\T) &= \Lop^\dagger \mathcal{M}_{0,\frac{2\pi}{P}}(\T) \oplus \mathrm{Span} \{1\}.
%\end{align}
Then, using Proposition~\ref{prop:y0forz0} (with a restriction of the search space which does not affect the proof), we have that $f_P^* = y_0 + \Lop^\dagger w_P^*$, where 
\begin{equation}
\label{eq:gridded_pb_measures}
    w_P^* \in 
    \underset{w \in \mathcal{M}_{0,\frac{2\pi}{P}}(\T)}{\arg \min}
    \frac{1}{2} \lVert \bm{y} - \bm{\nu}(\Lop^\dagger w + y_0) \rVert_2^2 + \lambda \lVert w \rVert_{\mathcal{M}}.
\end{equation}
     We now prove that the Radon measures $w_P^*$ converge towards the unique solution of
     \begin{equation}
     \label{eq:opti_measures}
       w^* =
    \underset{w \in \mathcal{M}_{0}(\T)}{\arg \min}
    \frac{1}{2} \lVert \bm{y} - \bm{\nu}(\Lop^\dagger w + y_0) \rVert_2^2 + \lambda \lVert w \rVert_{\mathcal{M}}
    \end{equation}
    for the weak* topology when $P\rightarrow \infty$, where the uniqueness of $w^*$ follows from \eqref{eq:reformulation_measures} in Proposition~\ref{prop:y0forz0}. This convergence is proved by following \cite[Proposition 4]{duval2017sparseI}; the fact that the search space in \eqref{eq:opti_measures} is $\Spc M_0(\T)$ rather than $\Spc M(\T)$ does not impact the proof.
    Then, the operator $\Op L^\dagger$ is linear and continuous between $\mathcal{M}_0(\T)$ and $\mathcal{M}_{\Op L}(\T)$ for their respective weak* topologies. This implies that $f_P^*$ converges to $f^*$ for the weak* topology over $\mathcal{M}_{\Op L}(\T)$. 
    According to~\cite[Proposition 9]{fageot2020tv}, $\Op L = \mathrm{D}^M$ is sampling-admissible for $M\geq 2$, which implies in particular that $\Sha$ is in the predual of $\mathcal{M}_{\Op L}(\T)$. Equivalently, this implies that $f \mapsto f(x)$ is weak*-continuous over $\mathcal{M}_{\Op L}(\T)$, which implies that $f_P^*(x) \rightarrow f^*(x)$ for any $x \in \T$ (pointwise convergence). 
    
   We now prove that the family $(f_P^*)_{P \in \N}$ is equicontinuous. Then, by~\cite[Theorem 15, Chapter 7]{kelley2017general}, pointwise and uniform convergences are equivalent, which will conclude the proof. Since $f_P^*$ is a $\Op L$-spline, using the expansion~\eqref{eq:spline_green}, we have
   \begin{equation}
   f_P^* = y_0 + \Op L^\dagger w_P^* = y_0 + \sum_{n=1}^{N_p} a_P[n] g_{\Op L}(\cdot - x_{P,n}), \quad x_{P,n} \in \left\{ \frac{2 \pi p}{P}, \ 0 \leq p \leq P-1 \right\}
   \end{equation}
   for some coefficients $a_P[n]$, $1 \leq n \leq N_p$, where $g_\Op L$ is the Green's function of $\Op L$ defined in~\eqref{eq:green_def}.
%    For any $x \in \T$, we have that
%    \begin{equation} \label{eq:unifbound}
%        |f_P^*(x)| \leq \sum_{n=1}^{N_p} |a_P[k]| |g_M(x-x_{P,n})| \leq \sum_{n=1}^{N_p} |a_P[k]| \lVert g_M \rVert_\infty = \mnorm{w_P^*} \lVert g_M \rVert_\infty.
%    \end{equation}
%    Moreover, we have seen thanks to~\cite{duval2017sparseI} that $(w_P^*)$ weak-* converges. It is therefore bounded for the total variation norm, thanks to the Uniform Boundedness Principle and the functions $f_P^*$ are uniformly bounded due to \eqref{eq:unifbound}
    Moreover, $g_\Op L$ is a periodic Lipschitz function for $\Op L = \Op D^M$ and $M \geq 2$, hence $\lVert g_\Op L \rVert_{\mathrm{Lip}} := \sup_{x,y \in \R, \ x \neq y} \frac{|g_\Op L(x) - g_\Op L(y)|}{|x-y|}  < \infty$.
    For any $x,y \in \R$, we have that
    \begin{align}
    \label{eq:uniform_lipschitz}
    |f_P^*(x) - f_P^*(y)| &\leq \sum_{n=1}^{N_p} |a_P[n]| |g_{\Op L}(x - x_{P,n}) - g_{\Op L}(y - x_{P,n})| \leq \left( \sum_{n=1}^{N_p} |a_P[n]| \right) \lVert g_\Op L \rVert_{\mathrm{Lip}} |x-y| \nonumber \\ &= \mnorm{w_P^*} \lVert g_\Op L \rVert_{\mathrm{Lip}}|x-y|.
    \end{align}
We have seen that $w_P^* \to w^*$ when $P \to \infty$ for the weak* topology. It is therefore bounded for the total-variation norm, thanks to the uniform boundedness principle.
We therefore deduce from \eqref{eq:uniform_lipschitz} that the $f_P^*$ are uniformly Lipschitz, and therefore equicontinuous, which proves the desired result.
\end{proof}

The first part of the proof of Theorem~\ref{thm:unifconv}, dealing with the pointwise convergence, mostly relies on the generalization of the weak* convergence studied in~\cite[Proposition 4]{duval2017sparseI}. Duval and Peyré use tools from $\Gamma$-convergence (see \cite{dal2012introduction} for an introduction) and are themselves inspired by~\cite{heinsnovel}.

Theorem~\ref{thm:unifconv} shows that our grid-based discretization yields spline solutions that are arbitrarily close to the unique solution $f^*$ of~\eqref{eq:penalized_pb_quadratic} in the uniform sense when the discretization step $h = \frac{2\pi}{P}$ vanishes. It leverages the uniqueness of the spline reconstruction from Fourier measurements ensured by Theorem \ref{theo:main}.

\section{B-spline-Based Algorithm}
\label{sec:algo}

We now introduce our proposed algorithm to solve the discretized Problem~\eqref{eq:discretized_continuous_pb} in an \emph{exact} way, \ie without any discretization error. The algorithm is based on \cite{Debarre2019} and uses the B-spline basis to represent the space of uniform splines $\MLh$. The main difference here with \cite{Debarre2019} is the periodic setting, which actually simplifies the treatment of the boundary conditions. Moreover, for the sake of conciseness, we focus here on discretizing for a fixed grid; we do not present the multiresolution aspect of the algorithm introduced in \cite{Debarre2019}, although it can seamlessly be adapted to our setting.

\subsection{Preliminaries on Uniform Periodic  Polynomial Splines}

A convenient feature of the space $\MLh$ is that it is generated by periodic B-splines, as will be proved in Proposition~\ref{prop:B-spline_representation}. To this end, we first provide some background information on B-splines and their periodized versions. B-splines are popular basis functions \cite{Schoenberg1973cardinal} that are widely used in signal processing applications \cite{de1978practical, Unser1999splines}, in part due to their short support which leads to well-conditioned optimization tasks. In the non-periodic setting, the scaled B-spline $\beta_{\Op L, h}$ of the operator $\Op L = \Op D^M$ with grid size $h > 0$ is characterized by its Fourier transform
\begin{align}
\widehat{\beta}_{\Op L, h}(\omega) = \frac{1}{h^{M-1}}  {\left( \frac{1 - \ee^{ \ii \omega h}}{\ii \omega} \right) }^M, \quad \forall \omega \in \R.
\end{align}
%where $\widehat{\beta_{n, h}}$ is the generalized Fourier transform of $\beta_{n, h}$
The scaled B-spline $\beta_{\Op L, h}$ is a piecewise polynomial of order $(M-1)$ with continuous $(M-2)$-th derivative, and is supported over the interval $[0, h M]$.

For any integer $P\geq 1$, the periodized $\Op L$ B-spline with grid size $h = \frac{2 \pi}{P}$ is then defined as 
\begin{align}
    \betaperN (x) = \sum_{k \in \Z} \beta_{\Op L, \frac{2\pi}P}(x - 2 \pi k), \quad \forall x \in \T, 
\end{align}
which is a converging sum due to the finite support of $\beta_{\Op L, \frac{2\pi}{P}}$. In fact, for $P \geq M$, the periodic B-spline is not aliased, since we have $\mathrm{Supp}(\beta_{\Op L, \frac{2\pi}P}) = [0, M \frac{2 \pi}{P}] \subset [0, 2\pi] = \TT$: we thus have $\betaperN = \beta_{\Op L, \frac{2\pi}P}$ in the interval $\TT$. Clearly, $\betaperN$ is $2 \pi$-periodic, and one readily shows from standard Fourier analysis that its Fourier series coefficients are given by
\begin{align}
    \betaperNhat [k] = \frac{1}{2 \pi} \widehat{\beta}_{\Op L, \frac{2\pi}P}(k) = P^{M-1}  {\left( \frac{1 - \ee^{ \ii k \frac{2\pi}P}}{2 \pi \ii k} \right) }^M.
\end{align}
Moreover, $\betaperN$ is a periodic $\Op L$-spline in the sense of Definition~\ref{def:Lspline}, and its innovation is given by
\begin{align}
\label{eq:innovation_B-spline}
    \Op L \betaperN = \left( \frac{P}{2 \pi} \right)^{M-1} \sum_{m=0}^M d_\Op L[m] \Sha (\cdot - m \frac{2 \pi}P),
\end{align}
where the $P$-periodic sequence $d_\Op L$ is characterized by its discrete Fourier transform (DFT) $D_\Op L[k] = (1 - \ee^{- \ii k \frac{2 \pi}P })^M$. %, and is thus supported in $\{0, \ldots , M\}$. 
This relation is easily verified in the Fourier domain. As an example, for $\Op L= \Op D$, $d_\Op L$ is the $P$-periodized finite-difference sequence $d_\Op L[k] = \delta_P[k] - \delta_P[k-1]$ where $\delta_P[k]$ is the $P$-periodized Kronecker delta sequence. As stated earlier, periodic B-splines share the same celebrated property as regular B-splines: they are generators of the space of uniform (periodic) splines $\MLh$ introduced in \eqref{eq:space_uniform_splines}.

\begin{proposition}
\label{prop:B-spline_representation}
The periodic B-spline $\betaperN$ is a generator of the space $\MLh$, \ie we have
\begin{align}
\label{eq:B-spline_representation}
    \MLh = \Big\{ f = \sum_{p=0}^{P-1} c[p] \betaperN \left(\cdot -  \frac{2 \pi p}P \right) , \ \V c = (c[0], \ldots , c[P-1]) \in \R^P \Big\}.
\end{align}
\end{proposition}
\begin{proof}
We first observe that the space $\MLh$ is a $P$-dimensional vector space: there are $(P-1)$ degrees of freedom for the $a[p]$ coefficients in \eqref{eq:space_uniform_splines} ($P$ coefficients and one linear constraint $\sum_{p=0}^{P-1} a[p] = 0$), and one for the mean $\widehat{f}[0]$ (see Proposition~\ref{prop:periodicTVused}). Next, we prove that for any $\V c = (c[0], \ldots , c[P-1]) \in \R^P$, we have $f = \sum_{p=0}^{P-1} c[p] \betaperN(\cdot -  \frac{2 \pi p}P) \in \MLh$. Indeed, we have
\begin{align}
\label{eq:innovation_uniform_spline}
    \Op L f &= \left( \frac{P}{2\pi} \right)^{M-1} \sum_{p=0}^{P-1} \sum_{m = 0}^M c[p] d_\Op L[m] \Sha \left(\cdot - (p+m) \frac{2 \pi}P \right) \nonumber \\
    &= \left( \frac{P}{2\pi} \right)^{M-1} \sum_{p=0}^{P-1} (\V d_\Op L \ast \V c)[p] \Sha \left(\cdot -  \frac{2 \pi p}P \right),
\end{align}
where \eqref{eq:innovation_B-spline} was used for the first line, and $(\V d_\Op L \ast \V c)$ denotes here the cyclic convolution between the vectors $\V d_\Op L = (d_\Op L[0], \ldots , d_\Op L[P-1])$ and $\V c$. This proves that $f \in \MLh$ with coefficients $(a[0], \ldots , a[P-1]) = \left( \frac{P}{2\pi} \right)^{M-1} (\V d_\Op L \ast \V c)$, and thus that the space generated by shifts of $\betaperN$ is included in $\MLh$. Yet both are $P$-dimensional vector spaces, which proves that they are in fact equal.
\end{proof}

\subsection{Discrete Problem Formulation}

In practice, to solve Problem~\eqref{eq:discretized_continuous_pb}, we use the B-spline representation~\eqref{eq:B-spline_representation} of $\MLh$. The choice of the B-spline representation is guided by numerical considerations: B-splines have the shortest support among any uniform $\Op L$-spline, and thus lead to well-conditioned optimization tasks. The problem thus consists in optimizing over the $c[0], \ldots , c[P-1]$ coefficients, which leads to a computationally feasible finite-dimensional problem, as demonstrated in the following proposition.
\begin{proposition}
Problem~\eqref{eq:discretized_continuous_pb} is exactly equivalent to solving the finite-dimensional problem
\begin{align}
\label{eq:discrete_pb}
\Spc W_{\lambda, P}(\V y) = \argmin_{\V c \in \R^P}   \frac{1}{2} \Vert \M H \V c - \V y \Vert_2^2 + \lambda \left( \frac{P}{2\pi} \right)^{M-1} \Vert \V d_\Op L \ast \V c \Vert_1,
\end{align}
where the matrix $\M H \in \C^{(K_c + 1) \times P}$ is given by
$ \M H_{k, \ell} = \nu_k \left(\betaperN(\cdot - \ell \frac{2 \pi}{P}) \right) = \ee^{- \ii \ell \frac{2 \pi}P } \betaperNhat[k]$, $\V d_\Op L = (d_\Op L[0], \ldots , d_\Op L[P-1])$, and $\V c = (c[0], \ldots , c[P-1])$.
The continuous-domain reconstructed signal is then $f = \sum_{p=0}^{P-1} c^\ast_p \betaperN(\cdot -  \frac{2 \pi p}P)$, where $\V c^\ast \in \Spc W_{\lambda, P}(\V y)$.
\end{proposition}
\begin{proof}
This equivalence is obtained by plugging in $f = \sum_{p=1}^P c^\ast_p \betaperN(\cdot -  \frac{2 \pi p}P) \in \MLh$ into the cost function of problem~\eqref{eq:penalized_pb_quadratic}. The expression of the system matrix $\M H$ immediately follows. The expression of the regularization term follows from \eqref{eq:innovation_uniform_spline} and the fact that $\Vert \sum_{p=0}^{P-1} a_p \Sha (\cdot - x_p) \Vert_\Spc M = \Vert \V a \Vert_1$ for pairwise-distinct knot locations $x_k$.
\end{proof}

Problem~\eqref{eq:discrete_pb} is a standard discrete problem with $\ell_1$ regularization, and a solution to the latter can be reached using proximal solvers such as ADMM \cite{boyd2010distributed}. 

\section{Experiments}\label{sec:experiments}

In this section, we present some results of our discretization method presented in the previous section in various experimental settings.

\subsection{Effect of Gridding}

\subsubsection{Qualitative effect}
We first present a toy experiment to illustrate the effect of gridding in our discretization method, \ie restricting the search space to $\MLh$. We therefore design an experiment in which the solution of the problem is known, in order to observe whether our algorithm is able to reconstruct it. To this end, we take $\Op L = \Op D^2$, and generate a ground-truth signal $f_0$ which is a periodic $\Op D^2$-spline with 2 knots (the locations and amplitudes of the knots are picked at random). We then compute the noiseless data vector $\V y = \V \nu(f_0)$ for $K_c = 3$, and solve the corresponding problem~\eqref{eq:penalized_pb_quadratic} with a small regularization parameter $\lambda = 10^{-7}$ in order to enforce the constraints $\V \nu(f) \approx \V y$ with very low error. Since the form of $f_0$ is compatible with that of the solution given by Theorem~\ref{theo:main}, the hope is that $f_0$ will be very close to the solution $f^\ast$ to problem~\eqref{eq:penalized_pb_quadratic}, which is confirmed by our experiments.

%Using the so-called vanishing derivatives precertificate \cite[Definition 6]{denoyelle2019sliding}, we can in fact prove that $f_0$ is a solution of the constrained problem~\eqref{eq:Lyf}, and thus that the solution to the penalized problem~\eqref{eq:penalized_pb_quadratic} is very close to $f_0$, and most likely has two knots as well.

In Figure~\ref{fig:P=16}, we show the reconstruction result of our algorithm, using a voluntarily coarse grid with $P=16$ points for visualization purposes. We observe that since the knot of $f_0$ are quite far from the grid, it is difficult to approximate $f_0$ with an element of $\MLh$. The reconstruction therefore requires several knots on the grid to mimic a single knot of $f_0$, and thus has a much higher sparsity ($N=7$ knots versus $N=2$ for $f_0$).

\begin{figure}[t]
\centering
\subfloat[Reconstruction with $P=16$ grid points (sparsity: $N = 7$ knots).]{\includegraphics[width=.5\linewidth]{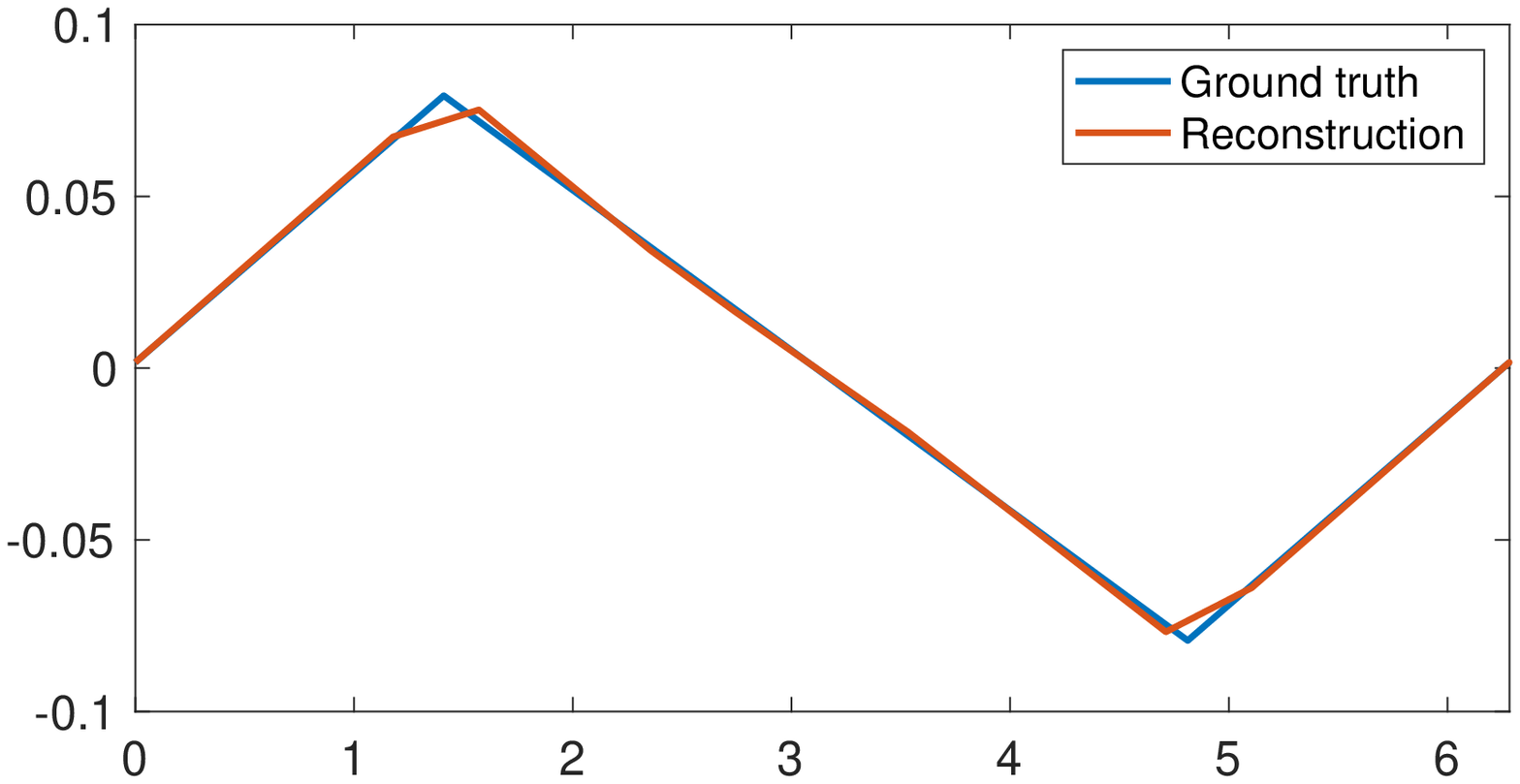} \label{fig:P=16}}
\subfloat[Reconstructions with different grid sizes around the knot at 1.41 ]{\includegraphics[width=.5\linewidth]{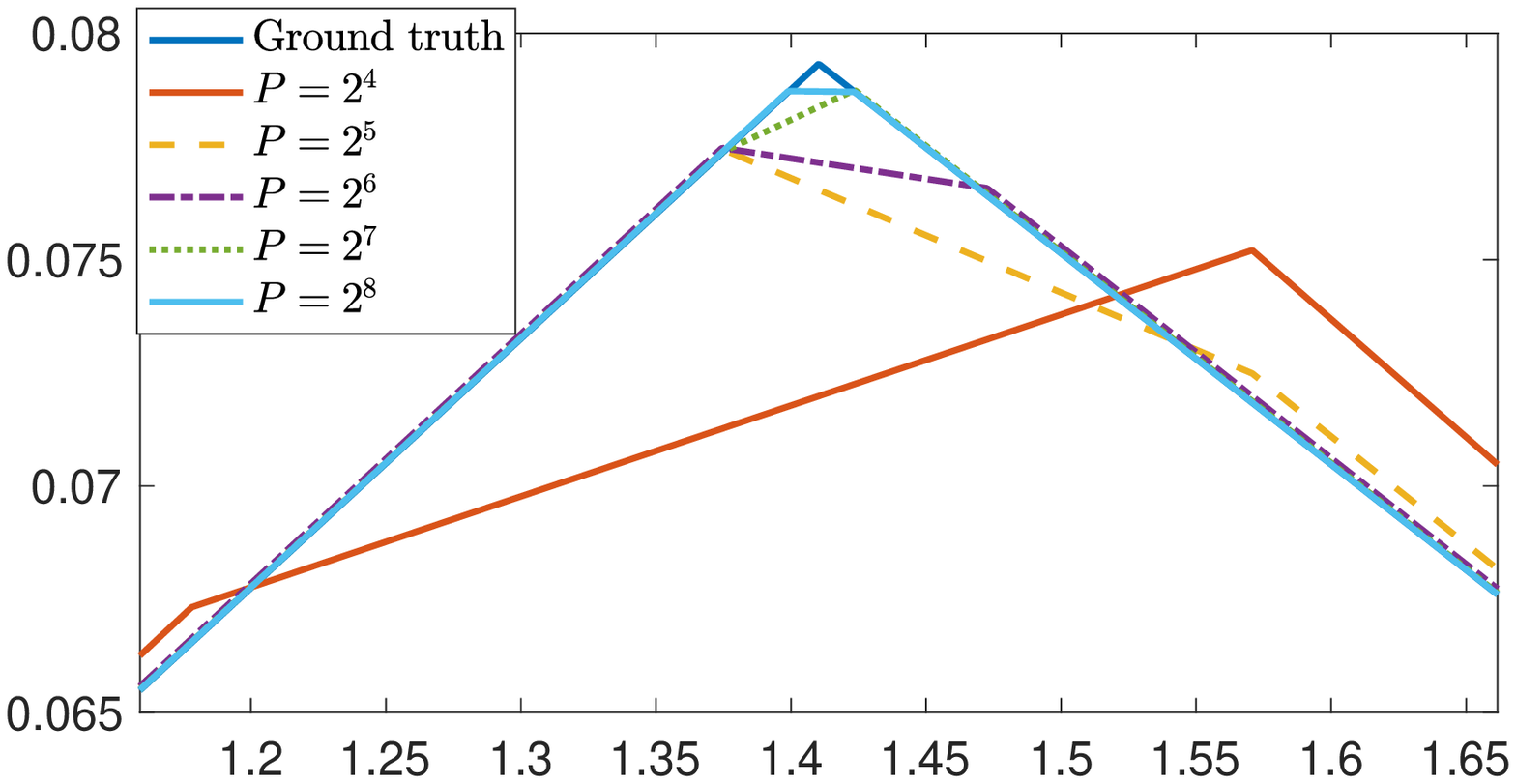} \label{fig:zoom_knot}}
\caption{Noiseless reconstruction of a piecewise-linear spline with $N = 2$ knots, $K_c = 3$, and $\lambda = 10^{-7}$.}
\label{fig:2knots}
\end{figure}

However, as we increase the number of grid points, the effect of gridding is greatly reduced, as illustrated in Figure~\ref{fig:zoom_knot}: with $P=512$, the reconstruction using our algorithm is visually indistinguishable from $f_0$ (which is why we do not show it). However, the knot locations of $f_0$ still do not exactly lie on the grid, and thus our reconstruction still requires multiple knots to mimic a single knot of $f_0$, which leads to a sparsity of $N=4$. Specifically, our reconstruction has two knots at consecutive grid points 1.3990 and 1.4113 mimicking the knot at 1.4103 of $f_0$, and two knots at 4.8106 and 4.8228 mimicking the knot at 4.8122 of $f_0$. This effect of knot multiplication due to gridding has already been observed and studied extensively in \cite{duval2017sparseI} in the absence of a regularization operator $\Op L$.

The conclusion is thus that gridding leads to visually near-perfect reconstruction when the number of grid points is very large, which is in line with Theorem~\ref{thm:unifconv}; however, the sparsity of the reconstruction is a poor indicator of the sparsity of the true solution of Problem~\eqref{eq:penalized_pb_quadratic}, since gridding induces clusters of knots.

\subsubsection{Quantitative effect}
In Theorem~\ref{thm:unifconv}, we have proved that any sequence of continuous-domain solutions $f^\ast_P$ to the grid-restricted problem converges uniformly towards the unique solution $f^\ast$ of problem~\eqref{eq:penalized_pb_quadratic} when $P$ goes to infinity. In order to quantify the speed of this convergence, using the same experimental setting as in Figure~\ref{fig:2knots}, we compute the error $\Vert f^\ast_P - f_0 \Vert_\infty$ where $f^\ast_P$ is the reconstructed signal using our grid-based algorithm, and the ground truth $f_0$ is a proxy for the solution $f^\ast$ to problem~\eqref{eq:penalized_pb_quadratic}. As explained earlier, this is a reasonable proxy due to the very small regularization parameter $\lambda = 10^{-7}$. In order to limit the effect of randomness in the choice of the knots of the ground truth, we apply a Monte Carlo-type method by generating 100 different ground truth signals (following the methodology described in the previous section) and averaging the error over these 100 runs. These average errors for different grid sizes are shown in Figure~\ref{fig:Linf_MC}. The trend appears to be linear in log-log scale, which indicates an empirical speed of convergence of $\Vert f^\ast_P - f_0\Vert_\infty \approx(\frac{C}{P^s})$ for some constant $C > 0$ and where $-s < 0$ is the slope of the linear function. We observe here that $s \approx 1$ with $s<1$. This is consistent with classical approximation theory results, since the approximation power of linear splines with grid size $h$ is in $\mathcal{O}(h)$ for the supremum norm, which corresponds to $s=1$. There is therefore no hope of having $s > 1$; our observation $s < 1$ can likely be attributed to the fact that we use $f_0$ as a proxy for $f^\ast$ and to increased numerical issues when the grid size decreases.

\begin{figure}[t]
\centering
\includegraphics[width=.5\linewidth]{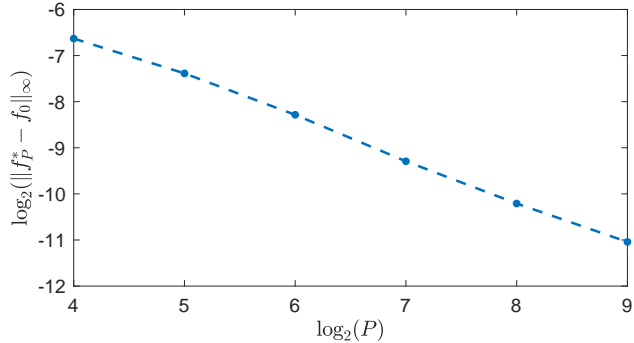}
\caption{Average error $\Vert f^\ast_P - f_0\Vert_\infty$ over 100 runs for different grid sizes $P$ (in log-log scale).}
\label{fig:Linf_MC}
\end{figure}

\subsection{Noisy Recovery of Sparse Splines}

In our next experiment, we attempt to recover a ground-truth $f_0$ signal based on noisy data $\V y$ with a regularization operator $\Op L = \Op D$. Once again, the ground-truth signal fits the signal model of problem~\eqref{eq:penalized_pb_quadratic}, \ie $f_0$ is a periodic $\Op D$-spline (piecewise constant signal) with $N=7$ knots. Each knot $x_n$ is chosen at random within consecutive intervals of length $\frac{2 \pi}{7}$, and the vector of amplitudes $\V a = (a_1, \ldots , a_n)$ is an i.i.d. Gaussian random vector projected on the space of zero-mean vectors.  The measurements are corrupted by some additive i.i.d. Gaussian noise\footnote{For complex entries, both the real and imaginary parts are i.i.d. Gaussian variables with the same $\sigma$.} $\V n \in \R \times \C^{K_c}$ with standard deviation $\sigma = 10^{-3}$, \ie $\V y = \V \nu (f_0) + \V n$.

\begin{figure}[t]
\centering
\includegraphics[width=\linewidth]{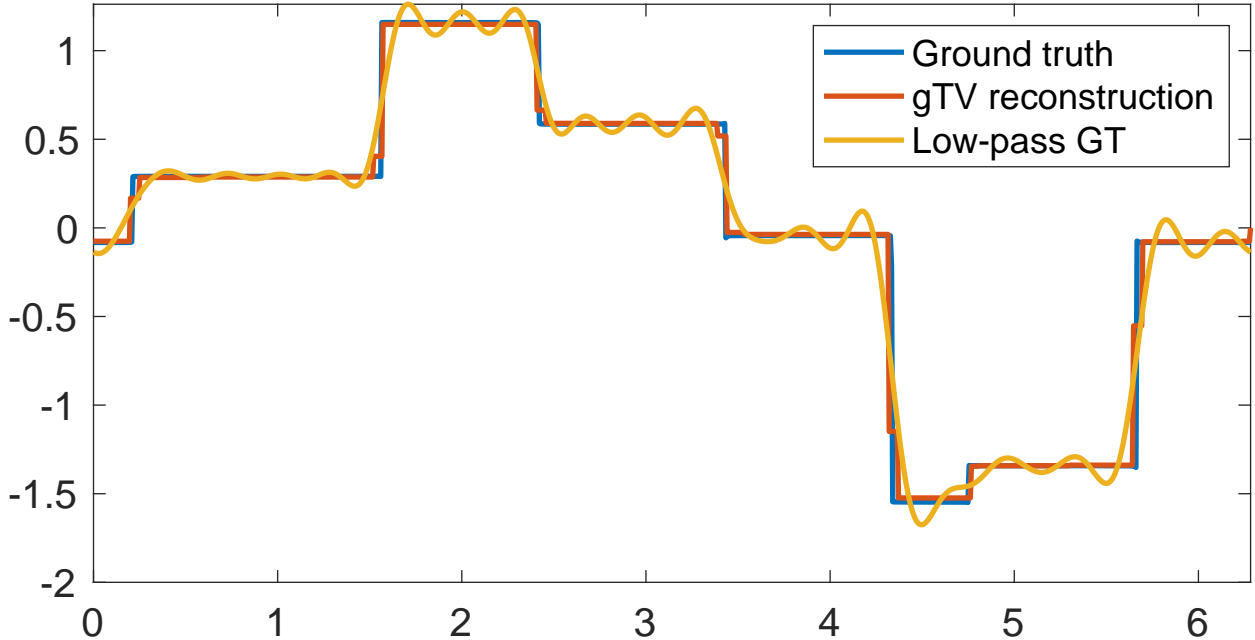}
\caption{Recovery of piecewise-constant spline with 7 knots with $K_c = 20$. For the our reconstruction (gTV), we use $\lambda = 10^{-2}$ and $P=256$ grid points; the sparsity of the reconstruction is $N=20$ knots. The data $\V y$ is noisy in the gTV case, whereas the noiseless data $\V \nu(f_0)$ is used for the low-pass reconstruction (partial Fourier series of the ground truth signal up to the cutoff frequency $K_c$). }
\label{fig:l2-comparison}
\end{figure}

The reconstructed signal using our algorithm is shown in Figure~\ref{fig:l2-comparison}. Despite the presence of noise, the reconstruction of the ground truth $f_0$ is almost perfect. As observed in the previous experiment, the sparsity of the reconstruction ($N=20$) is higher than that of the ground truth ($N=7$) due to clusters of knots. We compare our reconstruction to the truncated Fourier series of $f_0$ up the $K_c$, \ie $f_{K_c} = \sum_{k = -K_c}^{K_c} \widehat{f}_0[k] e_k$, which solely depends on the noiseless data vector $\V \nu(f_0)$. Without any prior knowledge, this is the simplest reconstruction one can think of based the available data $\V \nu(f_0) = (\widehat{f}_0[0], \ldots , \widehat{f}_0[K_c]$). As it turns out, $f_{K_c}$ is also the unique solution to the following constrained $L_2$-regularized problem
\begin{align}
\label{eq:constrained_L2}
    f_{K_c} = \argmin_{f: \V \nu(f) = \V \nu(f_0)} \Vert f \Vert_{L_2},
\end{align}
as demonstrated in~\cite[Theorem 3]{gupta2018continuous}. In fact, adding any LSI regularization operator $\Op L$ in~\eqref{eq:constrained_L2} still yields the same solution, since the basis functions $\varphi_m$ in~\cite[Theorem 3]{gupta2018continuous} span the same space. This is due to the fact that the measurement functionals $\nu_m$, \ie complex exponentials, are eigenfunctions of LSI operators.

As expected from the fact that $f_{K_c}$ is a trigonometric polynomial whereas $f_0$ has sharps jumps, the reconstruction is quite poor and exhibits Gibbs-like oscillations, despite the absence of noise. This clearly demonstrates the superiority of gTV over $L_2$ regularization for sparse periodic splines reconstruction. Note however that the gap in performance decreases as the order of $\Op L = \Op D^M$ increases, since Gibbs-like phenomena are less significant for smoother functions.

\section{Conclusion}
\label{sec:conclusion}
    
This paper deals with continuous-domain inverse problems, where the goal is to recover a periodic function from its low-pass measurements. The reconstruction task is formalized as an optimization problem with a TV-based regularization involving a high-order derivative operator. It was known that spline solutions always exist (representer theorem). Our main result has proved that the solution is in fact always unique. %, which is to the best of our knowledge the first result of the kind for TV-based inverse problems.
%{\color{blue}Moreover, we provided a new  sliding Frank-Wolfe algorithm reaching this unique solution, while proving its main theoretical properties and showing its practical relevance on simulations.}
We then studied the grid-based discretization of our optimization problem. We leveraged our uniqueness result to that any sequence of solutions of the discretized problems converge in uniform norm --- a remarkably strong form of convergence --- to the solution of the original problem when the grid size vanishes. Finally, we proposed a B-spline-based algorithm to solve the discretized problem, and we illustrated the relevance of our approach on simulations.

\section*{Acknowledgments}
The authors thank Shayan Aziznejad, Adrian Jarret, Matthieu Simeoni, and Michael Unser for interesting discussions.
Julien Fageot is supported by the Swiss National Science Foundation (SNSF) under Grant P400P2\_194364. The work of Thomas Debarre is supported by the SNSF under Grant 200020\_184646 / 1. Quentin Denoyelle is supported by the  European Research Council (ERC) under Grant 692726-GlobalBioIm.

% !TEX root = ../main.tex

 %\appendix 

% \section{Proof of Proposition \ref{prop:invertibleelipticL}}
% \label{app:tvboundcoeffs}

%         The Fourier sequence of $\Lop$ is also the one of $\Lop\{\Sha\} \in \Schp$, which is bounded by a polynomial (as any Fourier sequence of a periodic generalized function; see~\cite[Chapter VII]{Schwartz1966distributions}. This implies the existence of $\beta$ in~\eqref{eq:loweruperelliptic}. 
        
%         Next, by the same reasoning, the Fourier sequence of the pseudoinverse operator $\Lop^\dagger$ is also bounded by a polynomial. Yet for any $k \in \Z \setminus \{0 \}$, we have $\widehat{L}^\dagger[k] = 1/\widehat{L}[k]$, which implies the existence of $\alpha$ in~\eqref{eq:loweruperelliptic}. The inequality remains valid for $k = 0$, since $\widehat{L}^\dagger[0] = 0$.
        
%         Finally, we have $\Lop f = 0$ if and only if $\widehat{L}[k] \widehat{f}[k] = 0$ for any $k \in \Z$. For an elliptic operator $\Lop$, due to~\eqref{eq:loweruperelliptic}, this is equivalent to $\widehat{f}[k] = 0$ for any $k \neq 0$. Hence, the null space of $\Lop$ is $\Span \{1\}$.

{\footnotesize
\bibliographystyle{IEEEtran}
\bibliography{ms}}
\end{document}